%% file: max_lambda_max.tex
\documentclass{article}
\usepackage{amsthm}
\usepackage{bbm,verbatim}

\usepackage[a4paper,margin=27mm]{geometry}

\input{./shared/our_macros}

\numberwithin{equation}{section}
\usepackage{sectsty}
\sectionfont{\large}
\subsectionfont{\large}

\title{On Properties of Univariate Max Functions\\ at Local Maximizers}

\author{Tim Mitchell\thanks{
Max Planck Institute for Dynamics of Complex Technical Systems, Magdeburg, 39106 Germany, \href{mailto:mitchell@mpi-magdeburg.mpg.de}{\texttt{mitchell@mpi-magdeburg.mpg.de}}, ORCID: 0000-0002-8426-0242.}
\and
Michael L. Overton\thanks{Courant Institute of Mathematical Sciences, New York University, New York, NY 10012, USA. \href{mailto:mo1@nyu.edu}{\texttt{mo1@nyu.edu}}, ORCID: 0000-0002-6563-6371. 
Supported in part by U.S. National Science Foundation Grant DMS-2012250.}
}
\date{{\small\today}}

\theoremstyle{plain}
\newtheorem{theorem}{Theorem}[section]
\newtheorem{lemma}[theorem]{Lemma}
\newtheorem{corollary}[theorem]{Corollary}

\newtheorem{remark}[theorem]{Remark}

\def\plotex{0.445}
\def\scalederiv{0.34}

\begin{document}
\maketitle

\begin{abstract}
\input{./shared/abstract}
\end{abstract}


\input{main}

\footnotesize
\bibliographystyle{alpha} 
\bibliography{./shared/csc}
\end{document}

%% file: shared/our_macros.tex
\usepackage{amsmath,amssymb,amsfonts}
\usepackage{amsopn}

\usepackage{graphicx}
\usepackage{epstopdf}

\usepackage{hyperref}

\usepackage{subcaption}

\usepackage{booktabs}
\usepackage{multirow}
\usepackage{dcolumn}
\usepackage[detect-all]{siunitx}

\usepackage{enumitem}

\usepackage{mathtools}

\usepackage{newfloat}
\usepackage{algorithm}
\usepackage{algorithmic}

\usepackage{soul} 
\usepackage{color}

\ifpdf
  \DeclareGraphicsExtensions{.eps,.pdf,.png,.jpg}
\else
  \DeclareGraphicsExtensions{.eps}
\fi

\def\matlab{MATLAB}

\def\R{\mathbb{R}}
\def\C{\mathbb{C}}

\def\H{\mathbb{H}}

\def\bigO{\mathcal{O}}
\def\imagunit{\mathrm{i}} 
\def\transsym{\mathsf{T}}
\newcommand{\tp}[1][]{^{{#1}\transsym}}

\DeclarePairedDelimiter{\ceil}{\lceil}{\rceil}

\def\e{\mathrm{e}}
\def\eit{\eix{\theta}}
\def\emit{\emix{\theta}}

\newcommand{\eix}[1]{\e^{\imagunit #1}}
\newcommand{\emix}[1]{\e^{-\imagunit #1}}

\def\eps{\varepsilon}
\def\smin{\sigma_{\min}}

\def\Hcal{\mathcal{H}}
\def\Hinf{\Hcal_\infty}

\def\Ctwo{\Ccal^2}

\renewcommand{\Re}{\mathrm{Re}\,}
\renewcommand{\Im}{\mathrm{Im}\,}

\def\beq{\begin{equation}}
\def\eeq{\end{equation}}
\def\beqs{\begin{equation*}}
\def\eeqs{\end{equation*}}
\def\bseq{\begin{subequations}}
\def\eseq{\end{subequations}}

\DeclareFloatingEnvironment[placement=!ht]{algfloat}

\def\dom{\mathcal{D}}
\def\jset{\mathcal{I}}
\def\gmax{f_\mathrm{max}}

\def\l{\lambda}
\def\lmax{\l_\mathrm{max}}
\def\lmin{\l_\mathrm{min}}

\def\smax{\sigma_\mathrm{max}}
\def\smin{\sigma_\mathrm{min}}
\def\srad{\rho}
\def\sradin{\rho_\mathrm{in}}

\def\opt{x}

\def\Ck{\mathcal{C}^q}
\def\Czero{\mathcal{C}^0}
\def\Cone{\mathcal{C}^1}
\def\Ctwo{\mathcal{C}^2}
\def\Cthree{\mathcal{C}^3}
\def\Cfour{\mathcal{C}^4}
\def\Cq{\mathcal{C}^{q}}
\def\Cinf{\mathcal{C}^\infty}

\DeclareMathOperator{\diag}{diag}
\def\H{\mathbb{H}}

%% file: shared/abstract.tex
More than three decades ago, Boyd and Balakrishnan established a regularity result for the two-norm of a
transfer function at maximizers.  Their result extends easily to the statement that
the maximum eigenvalue of a univariate real analytic Hermitian matrix family is twice continuously differentiable, 
with Lipschitz second derivative, at all local maximizers, a property that is useful in several applications that we describe.
We also investigate whether this smoothness property extends to max functions more generally. We show that 
the pointwise maximum of a finite set of $q$-times continuously differentiable univariate functions must have zero derivative at a maximizer for $q=1$, but arbitrarily close to the maximizer, the derivative may not be defined,
even when $q=3$ and the maximizer is isolated.

%% file: main.tex

\section{Introduction}

Let $\H^{n}$ denote the space of $n\times n$ complex Hermitian matrices, 
let $\dom \subseteq \R$ be open, and
let \mbox{$H : \dom \to \H^{n}$} denote an analytic Hermitian matrix family in one real variable, i.e., for all~$\opt\in\dom$ and all $i,j\in\{1,\ldots,n\}$, there exist coefficients
$a_{0},a_{1},a_{2},\ldots$ such that the power series
$\sum_{k=0}^{\infty}a_{k}(t-\opt)^{k}$ converges to $H_{ij}(t)=\overline{H_{ji}(t)}$ for all $t$
in a neighborhood of~$\opt$.
For a generic family~$H$, the eigenvalues of $H(t)$ are simple for all $t \in \dom$; 
often known as the von Neumann-Wigner crossing-avoidance
rule~\cite{vonW93}, this phenomenon is emphasized in \cite[section~9.5]{Lax07},
where it is also illustrated on the front cover. 
The reason is simple: the real codimension of the subspace of Hermitian matrices with an eigenvalue of multiplicity $m$ is $m^{2}-1$, so to
obtain a double eigenvalue one would need three parameters generically; when the matrix family is real symmetric, the analogous codimension
is $\tfrac{m(m+1)}{2}-1$, so one would need two parameters generically.
When there are no multiple eigenvalues, the ordered eigenvalues of $H(t)$, say, $\mu_{j}(t)$ for $j=1,\ldots,n$, are all real
analytic functions. 

Let $\lmax:\H^{n}\to\R$ and $\lmin:\H^{n}\to\R$ denote
algebraically largest and smallest eigenvalue, respectively.
In the absence of multiple eigenvalues, $\lmax \circ H$ and $\lmin \circ H$ are both smooth functions of $t$.
However, for the nongeneric family $H(t)=\diag(t,-t)$, a double eigenvalue occurs at~$t=0$. 
By a theorem of Rellich, given in Section~\ref{sec:eigs},
the eigenvalues can be written as two real analytic functions, $\mu_{1}(t)=t$ and $\mu_{2}(t)=-t$, but we must give up the property that
these functions are ordered near zero. Consequently, the function $\lmax\circ H$ is not differentiable at its minimizer $t=0$.

In contrast, the function $\lmax\circ H$ is \emph{unconditionally} $\Ctwo$, i.e., twice continuously differentiable, with Lipschitz
second derivative, near all its local \emph{maximizers}, regardless of eigenvalue multiplicity at these maximizers. As we explain below, 
this observation is a straightforward extension of a well-known result of Boyd and Balakrishnan \cite{BoyB90} established more than 
three decades ago. One purpose of this paper is to bring attention to the more general result, as it is useful in a number of applications.
We also investigate whether this smoothness property extends to max functions more generally.  
We show that 
the pointwise maximum of a finite set of continuously differentiable univariate functions must have zero derivative at a maximizer.  However, arbitrarily close to the maximizer, the derivative may not be defined,
even if the functions are three times continuously differentiable and the maximizer is isolated.

\section{Properties of max functions at local maximizers}
\label{sec:maxfuns}
Let $\dom \subset \R$ be open, $\jset = \{1,\ldots,n\}$, and $f_j : \dom \to \R$ be continuous for all~$j \in \jset$,
and define
\begin{equation}
	\gmax(t) \coloneqq \max_{j \in \jset} f_j(t).
\end{equation}

\begin{lemma}
\label{lem:extrema}
Let $\opt \in \dom$ be any local maximizer of $\gmax$ with $\gmax(\opt)=\gamma$
and let $\jset_\gamma = \{ j \in \jset : f_j(\opt) = \gamma\}$.  Then
\begin{enumerate}[label=(\roman*),font=\normalfont,leftmargin=0.75cm]
\item for all $j \in \jset_\gamma$, $\opt$ is a local maximizer of $f_j$ and 
\item for all $j \in \jset \setminus \jset_\gamma$, $f_j(\opt) < \gamma$.
\end{enumerate}
\end{lemma}
\noindent
We omit the proof as it is elementary.  

We now consider adding additional assumptions on the smoothness of the~$f_j$,
writing $f_{j}\in\Cq$ to mean $f_{j}$ is $q$-times continuously differentiable.
Clearly, assuming that the $f_j$ are $\Czero$ is not sufficient to obtain differentiability at maximizers 
(e.g., $\gmax(t) = f_1(t) = -|t|$),
but $\Cone$ is sufficient.

\begin{theorem}
\label{thm:diff_at_max}
Let $\opt \in \dom$ be any local maximizer of $\gmax$  with $\gmax(\opt)=\gamma$.
Suppose that for all $j \in \jset$, 
$f_j$ is~$\Cone$ at~$\opt$.
Then $\gmax$ is differentiable at $\opt$ with $\gmax^\prime(\opt) = 0$.
\end{theorem}
\begin{proof}
Since the functions $f_j$ are continuous, clearly $\gmax$ is also continuous,
and without loss of generality, we can assume that $\gamma=0$.
Suppose that~$\gmax^\prime$ does not exist at~$\opt$ or does not equal zero, i.e.,
there exists some sequence~$\{\eps_k\}$ with $\eps_k \to 0$ such that $\lim_{k \to \infty} \frac{\gmax(\opt + \eps_k)}{\eps_k}$
does not exist or is not zero.
Since $\jset$ is finite, there exist a $j \in \jset$ and a subsequence $\{\eps_{k_\ell}\}$ such that 
\mbox{$\gmax(\opt + \eps_{k_\ell}) = f_j(\opt + \eps_{k_\ell})$} for all $k_\ell$, which implies that $f_j^\prime$ 
either does not exist or is not zero at $\opt$. 
However, as $f_j$ is $\Cone$ and with local maximizer $\opt$
by Lemma~\ref{lem:extrema},
it must be that $f_j^\prime(\opt) = 0$; hence, we have a contradiction.
\end{proof}

Assuming that the $f_j$ are $\Cone$ at (or near) a maximizer 
is not sufficient to obtain that $\gmax$ is twice differentiable at this point. 
For example, if 
\[
	f_1(t) = 
	\begin{cases}
	-t^2 & \text{if $ t \leq 0$} \\
	-3t^2 & \text{if $t > 0$}
	\end{cases}
	\qquad \text{and} \qquad 
	f_2(t) = -2t^2,
\]
then the second derivative of $\gmax=\max(f_1,f_2)$ does not exist at the maximizer $t=0$, as
\mbox{$\gmax^\prime(t) = -2t$} on the left and $-4t$ on the right, so 
$\lim_{t \to 0} \tfrac{\gmax^\prime(t)}{t}$ does not exist at $t=0$.
In this example, $\gmax$ is continuously differentiable at $t=0$,
but this does not hold in general,
even when assuming that the $f_j$ are $\Cthree$ near a maximizer; 
see Remark~\ref{rem:notC1} below. However, we 
do have the following result.

\begin{theorem}
\label{thm:tcd}
Let $\opt \in \dom$ be any local maximizer of $\gmax$  with $\gmax(\opt)=\gamma$.
Suppose that for all $j \in \jset$, 
$f_j$ is~$\Cthree$ near~$\opt$.
Then for all sufficiently small $|\eps|$,
\begin{equation}
	\label{eq:approx_quad}
	\gmax(x+\eps) = \gamma + M \eps^2 + O(|\eps|^3),
\end{equation}
where $M =  \tfrac{1}{2}\left(\max_{j \in \jset_\gamma} f_j^{\prime\prime}(\opt) \right) \leq 0$.
If the $\Cthree$ assumption is reduced to $\Ctwo$, then $\gmax(x+\eps) = \gamma + \bigO(\eps^2)$. 
\end{theorem}

\begin{proof}
Let $\gamma = \gmax(\opt)$ and let $\jset_\gamma = \{ j \in \jset : f_j(\opt) = \gamma \}$.
By Lemma~\ref{lem:extrema}, we have that $\opt$ is also a local maximizer of $f_j$ for all $j \in \jset_\gamma$
and $f_j(\opt) < \gamma$ for all $j \in \jset \setminus \jset_\gamma$.
Since the $f_j$ are Lipschitz near $\opt$,
\[
	\gmax(\opt+\eps) = \max_{j \in \jset_\gamma} f_j(\opt+\eps)
\]
holds for all sufficiently small $|\eps|$.
For each $j \in \jset_\gamma$, by Taylor's Theorem
we have that
\begin{align*}
	f_j(\opt+\eps) 
	&= f_j(\opt) + f_j^\prime(\opt)\eps + \tfrac{1}{2} f_j^{\prime\prime}(\opt) \eps^2 
			+ \tfrac{1}{6}  f_j^{\prime\prime\prime}(\tau_j) \eps^3 \\
	&= \gamma + \tfrac{1}{2} f_j^{\prime\prime}(\opt) \eps^2 + O(|\eps|^3)
\end{align*}
for $\tau_j$ between $x$ and $x + \eps$.
Taking the maximum of the equation above over all $j \in \jset_\gamma$ yields~\eqref{eq:approx_quad}.
The proof for the $\Ctwo$ case follows analogously.
\end{proof}

\begin{remark}
\label{rem:notC1}
Even with the $\Cthree$ assumption, 
$\gmax$ is not necessarily continuously differentiable at maximizers,
let alone twice differentiable.
A simple counterexample is given by 
$f_1(t) = t^8 (\sin(\tfrac{1}{t}) - 1)$ and $f_2(t) = t^8 (\sin(\tfrac{1}{2t}) - 1)$, with 
$f_1(0) = f_2(0) = 0$, where $f_1$ and $f_2$ are $\Cthree$ but not $\Cfour$ at $t=0$.
However, in this case the maximizer~$t=0$ of $\gmax$ is not an isolated maximizer.
In contrast, in Section~\ref{sec:counter}, we construct a counterexample where the $f_j$ are $\Cthree$ functions,
and for which $\gmax$ has an isolated maximizer, yet the derivative of $\gmax$ does not exist at points arbitrarily close to this maximizer.
It seems that this counterexample can be extended to apply to $\Ck$ functions for any $q\geq1$.
The key point of both of these counterexamples is not that the $f_j$ are insufficiently smooth \emph{per se}, but that the $f_j$ cross each other infinitely many times near maximizers.
\end{remark}

In light of Remark~\ref{rem:notC1}, we now make a much stronger assumption.

\begin{theorem}
\label{thm:c2_lipschitz}
Given a maximizer $\opt$ of $\gmax$, suppose there exist $j_1,j_2 \in \jset$, possibly equal,
such that, for all sufficiently small $\eps > 0$, $\gmax(\opt - \eps) = f_{j_1}(\opt - \eps)$ 
and $\gmax(\opt + \eps) = f_{j_2}(\opt + \eps)$,
with $f_{j_1}$ and $f_{j_2}$ both $\Cthree$ near $\opt$.
Then $\gmax$ is twice continuously differentiable,
with Lipschitz second derivative, near~$\opt$.
\end{theorem}
\begin{proof}
It is clear that $f_{j_1}(\opt)=f_{j_2}(\opt)=\gamma$ and $f_{j_1}^\prime(\opt)=f_{j_2}^\prime(\opt) = 0$.
By Theorem~\ref{thm:tcd}, both
$f_{j_1}^{\prime\prime}(\opt)$ and $f_{j_2}^{\prime\prime}(\opt)$ are equal to $M$, 
so $\gmax$ is locally described by two $\Cthree$ pieces whose 
function values and first and second derivatives agree at $\opt$. 
Hence, $\gmax$ is $\Ctwo$ with Lipschitz second derivative near~$\opt$.
\end{proof}

The assumptions of Theorem~\ref{thm:c2_lipschitz} hold if the $f_{j}$ are real analytic~\cite[Corollary~1.2.7]{KraP02}.
In particular, this holds if
the $f_{j}$ are eigenvalues of a univariate real analytic Hermitian matrix function, as we discuss in Section~\ref{sec:eigs}. First, we
present the $\Cthree$ counterexample mentioned above.

\section{An example with $\Cthree$ functions $f_j$ and an isolated maximizer for which $\gmax$ is not continuously differentiable at $\opt$}
\label{sec:counter}
Let $l_k = \tfrac{1}{2^k}$,
and 
$f_1 : [-1, 1] \to \R$ be defined by
\begin{equation}
	\label{eq:f1}
	f_1(t) = \begin{cases}
	p_k(t) & \text{if $t \in \left[l_{k+1}, l_k \right]$ for $k = 0,1,2,\ldots$} \\
	-t^2 	& \text{if $t \in [-1,0]$}  \\
	\end{cases}
\end{equation}
where  $p_k$ is a (piece of a) degree-nine polynomial chosen such that at
\begin{enumerate}
\item $l_{k+1}$ (the left endpoint), 
	$p_k$ and $-t^2$ agree up to and including their respective third derivatives,
\item $l_k$ (the right endpoint), 
	$p_k$ and $-t^2$ agree up to and including their respective third derivatives, 
\item $t_k = \tfrac{1}{2}(l_{k+1} + l_k)$ (the midpoint), 
	 $p_k$ and $-t^2$ agree, but the first derivative of $p_k$ is
	 $s_k \neq 1$ times the value of the first derivative of $-t^2$.
\end{enumerate}
\begin{figure}[t]
	\begin{subfigure}{.5\textwidth}
		\centering
		\includegraphics[scale=\plotex,trim=0mm 0mm 0mm 0mm,clip]{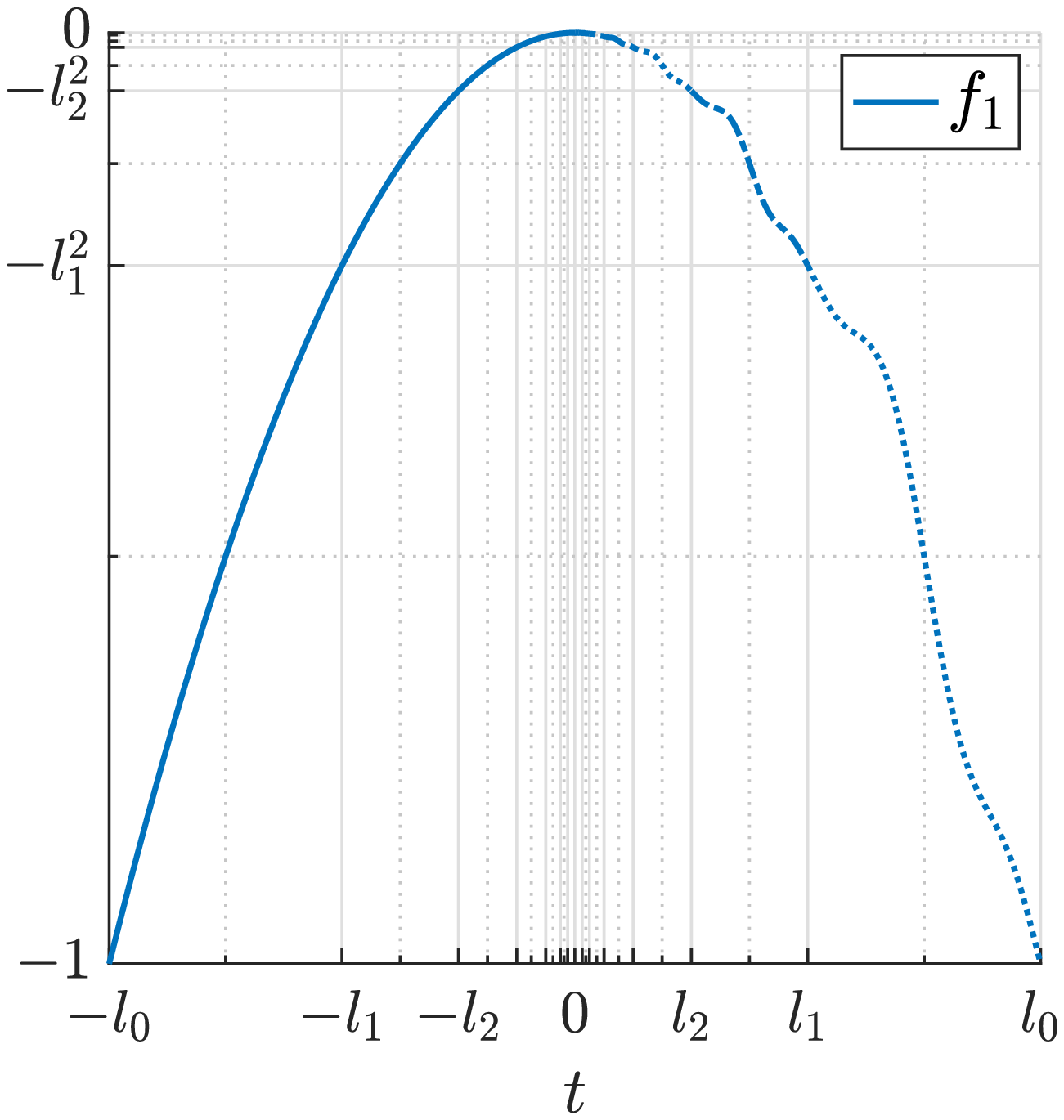}
		\caption{}
		\label{fig:f1}
	\end{subfigure}
	\begin{subfigure}{.5\textwidth}
		\centering
		\includegraphics[scale=\plotex,trim=0mm 0mm 0mm 0mm,clip]{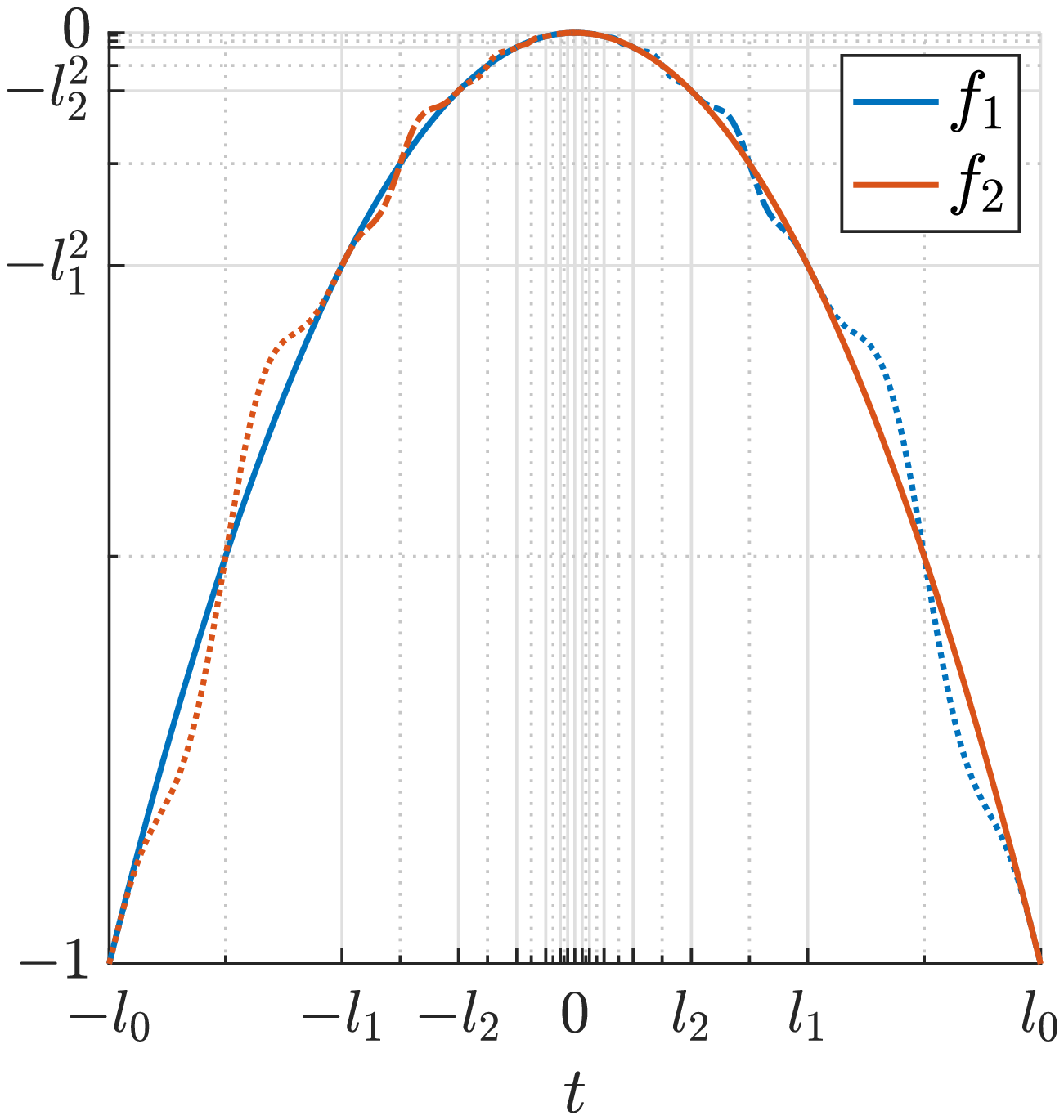}
		\caption{}
		\label{fig:counter}
	\end{subfigure}
\caption{Plots of $f_1$ and $f_2$ with $s_k = 2$; their $-t^2$ parts are shown in solid, while their $p_k$ 
	parts are shown in dotted for $k$ even and dash-dot for $k$ odd.}
\label{fig:f1_plots}
\end{figure}
For any $k$, the degree-nine polynomial $p_k$ is uniquely determined 
by the ten algebraic constraints given above.  If we choose $s_k = 1$,
then $p_k$ is simply $-t^2$.  However, by choosing $s_k > 1$ 
but sufficiently close to 1, then $p_k$ must be 
strictly decreasing between its endpoints $l_{k+1}$ and $l_k$ and cross
$-t^2$ at~$t_k$.
If this is done for all $k$, it follows that $t=0$ must be an isolated maximizer of~$f_1$.
See Figure~\ref{fig:f1} for a plot of~$f_1$ with~$s_k=2$ for all $k$; the choice $s_{k}=2$ is not close to 1 but was chosen
to make the features of $f_{1}$ easily seen.

Now define  $f_2(t) = f_1(-t)$, i.e., the graph of $f_{2}$ is a reflection of the graph of~$f_{1}$ across the vertical line $t=0$.
Figure~\ref{fig:counter} shows $f_1$ and $f_2$ plotted together, again with $s_{k}=2$, showing how they cross at every~$t_k$.
Recall that by our construction, their respective first three derivatives match at each~$l_k$, 
but their first derivatives do not match at any $t_k$. Figure~\ref{fig:derivs} shows plots of the 
first three derivatives of~$f_1$ for 
two different sequences $\{s_k\}$ respectively defined by $s_k = 1 + 2^{-k}$ and $s_k = 1 + 2^{-2k}$.
The rightmost plots in Figure~\ref{fig:derivs}
indicate that the first choice for sequence $\{s_k\}$ does not converge to 1 fast enough 
for $f_1^{\prime\prime\prime}$ to exist and be continuous at $t=0$,
but that the second sequence does.
In fact, for this latter choice of sequence, we have the following pair of theorems
respectively proving that $f_1$ is indeed $\Cthree$ with $t=0$ being an isolated maximizer.
We defer the proofs to Appendix~\ref{apdx:proofs} as they are a bit technical,
and  in Appendix~\ref{apdx:sk}, we discuss why $s_k = 1 + 2^{-k}$ does not converge to 1 sufficiently fast for $f_1^{\prime\prime\prime}(0)$ to exist.

\begin{theorem}
\label{thm:f1_c3}
For $f_1$ defined in~\eqref{eq:f1}, if $s_k = 1 + 2^{-2k}$,
then $f_1$ is $\Cthree$ on its domain $[-1,1]$.
\end{theorem}
\begin{theorem}
\label{thm:f1_isolated}
For $f_1$ defined in~\eqref{eq:f1}, if $s_k = 1 + 2^{-2k}$, then $t=0$ is an isolated maximizer of $f_1$, as well as an isolated
maximizer of $\gmax=\max(f_{1},f_{2})$.
\end{theorem}

Theorem~\ref{thm:diff_at_max} shows that $\gmax = \max(f_{1},f_{2})$ is differentiable at~$t=0$ with $\gmax^\prime(0) = 0$.
However, even though $f_1$ and $f_2$ are $\Cthree$ and  $t=0$ is an isolated maximizer of $\gmax$
with the choice of $s_k = 1 + 2^{-2k}$,
by construction, we have that (i) $t_k \to 0$ as $k \to \infty$, and
(ii) $\gmax$ is nondifferentiable at every $t_k$.
Hence, although $\gmax$ is differentiable at $t=0$, 
it is not $\Cone$ at this point, let alone twice differentiable.
Plots of $\gmax$ and its first and second derivatives are shown in Figure~\ref{fig:gmax},
where we see the discontinuities in $\gmax^{\prime}$ for all $t_k$ and $-t_k$.

\begin{remark}
For any $q \geq 1$, it seems that the same argument
extends to show that $\gmax$ is not necessarily $\Cone$ at $t=0$ when defined by functions~$f_j$ 
that are $\Ck$, using polynomials $p_k$ of degree $2q+3$.
From computational investigations for $q \in \{1,2,3,4,5\}$,
we conjecture that $s_k = 1 + 2^{-(k+1)}$ for $q=1$ and $s_k = 1 + 2^{-(q-1)k}$ for $q\geq2$ 
are suitable choices in general 
to obtain that $f_1$ is $\Ck$ with $t=0$ being an isolated maximizer.
It is not clear how to extend such an argument to the $\Cinf$ case.
\end{remark}

\begin{figure}[t]
	\begin{subfigure}{\textwidth}
		\centering
		\includegraphics[scale=\scalederiv,trim=0mm 0mm 0mm 0mm,clip]{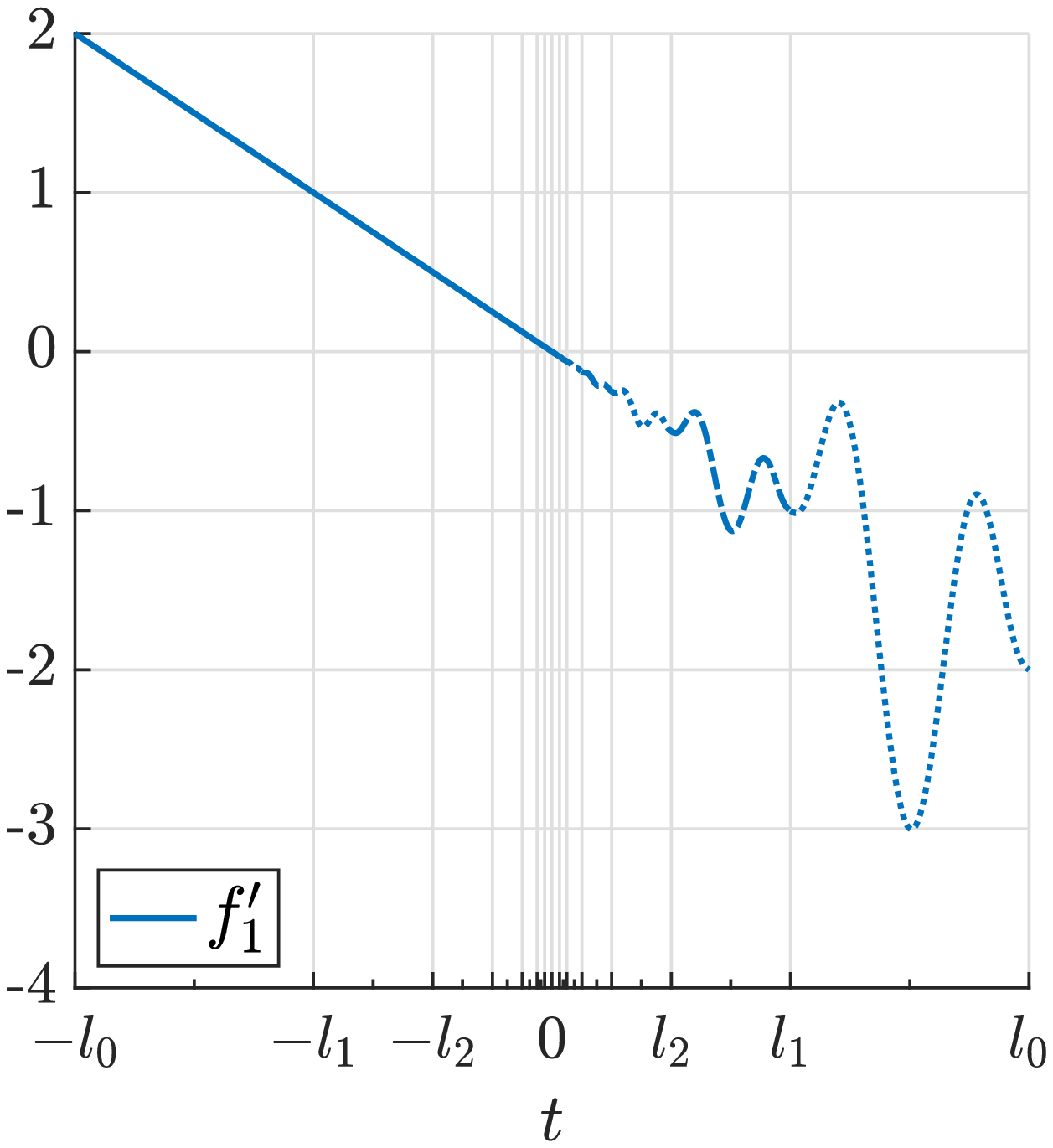}
		\includegraphics[scale=\scalederiv,trim=0mm 0mm 0mm 0mm,clip]{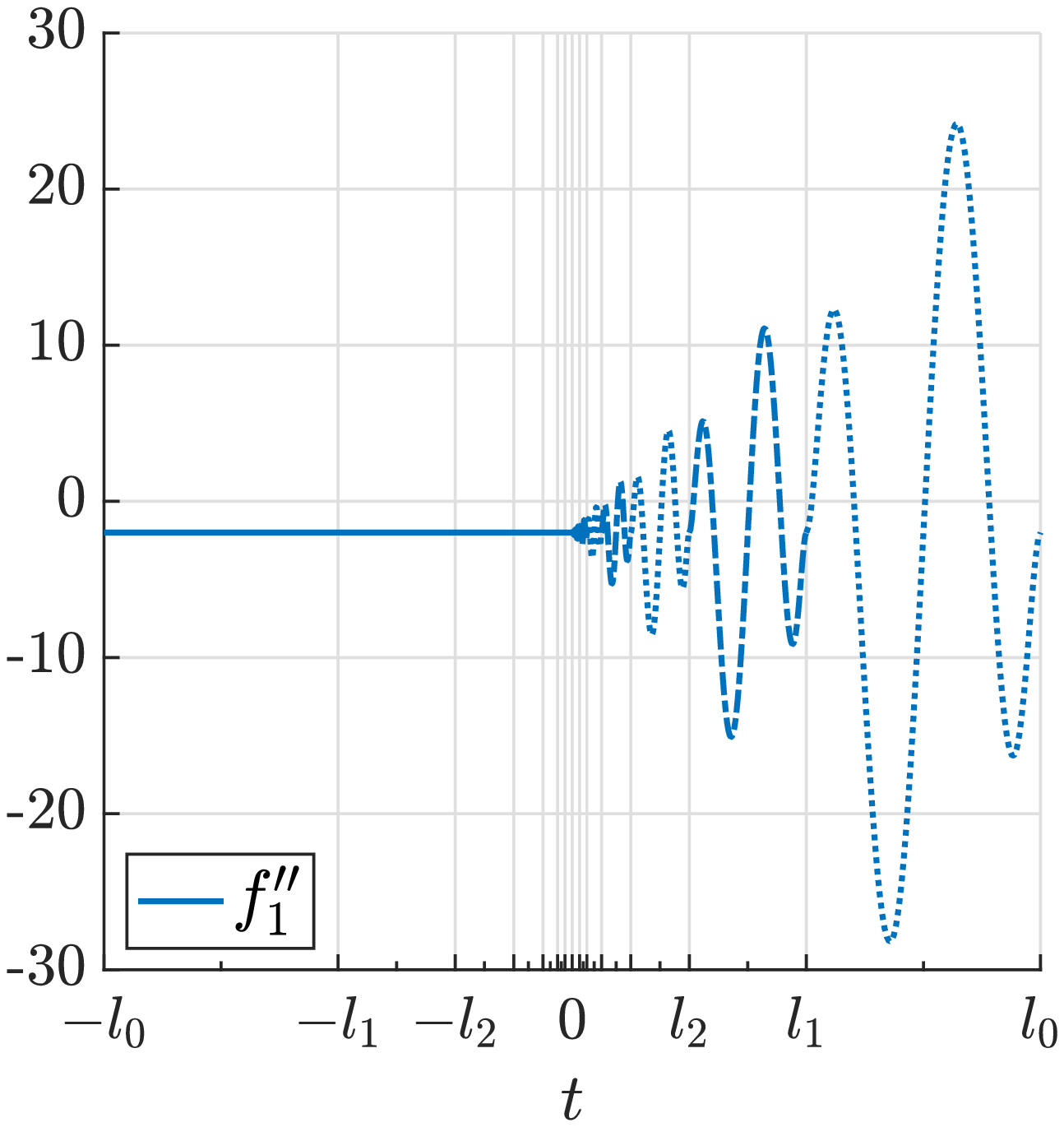}
		\includegraphics[scale=\scalederiv,trim=0mm 0mm 0mm 0mm,clip]{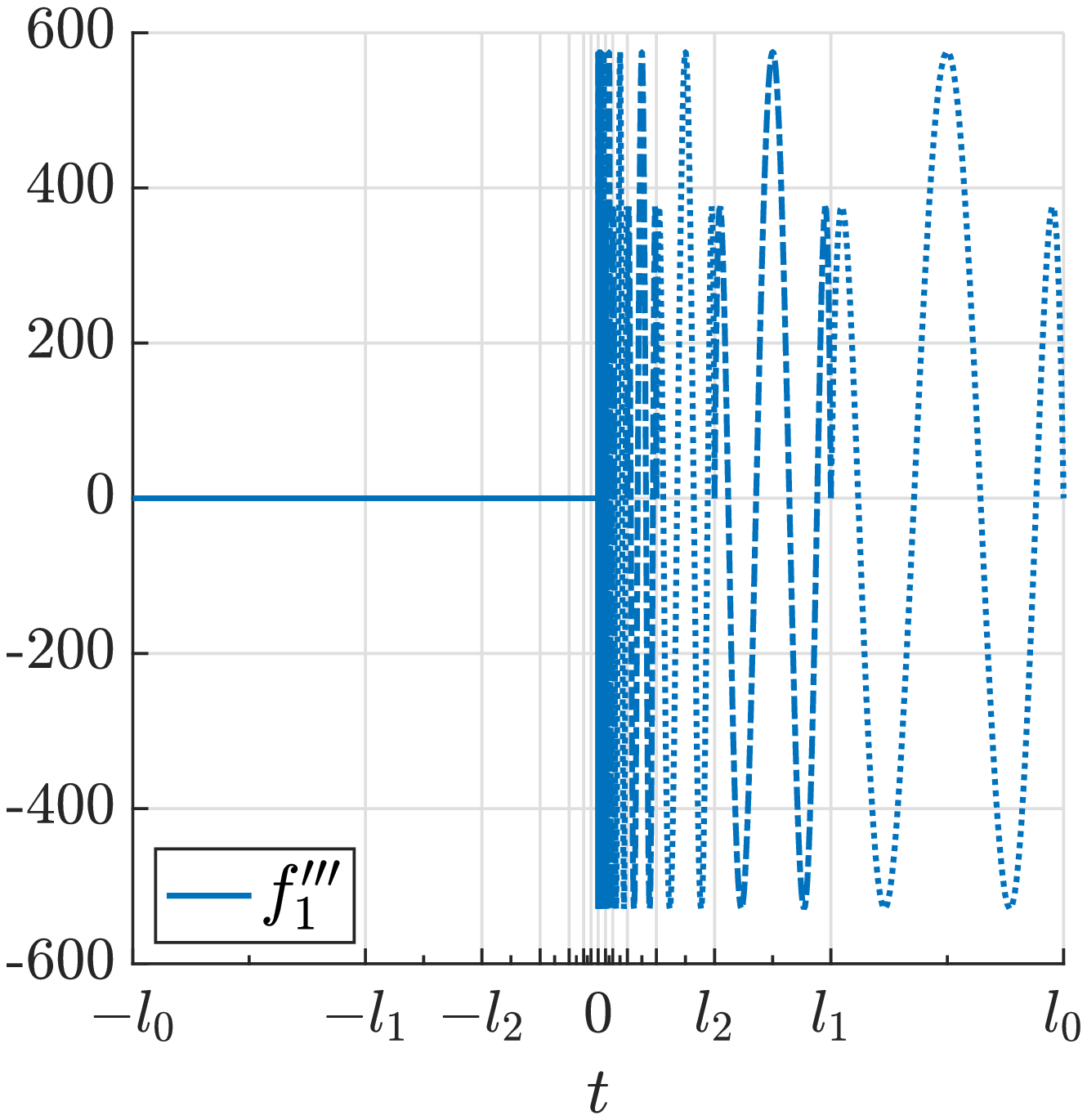}
		\caption{$s_k = 1 + 2^{-k}$}
		\label{fig:fppp_notC3}
	\end{subfigure}
	\begin{subfigure}{\textwidth}
		\centering
		\includegraphics[scale=\scalederiv,trim=0mm 0mm 0mm 0mm,clip]{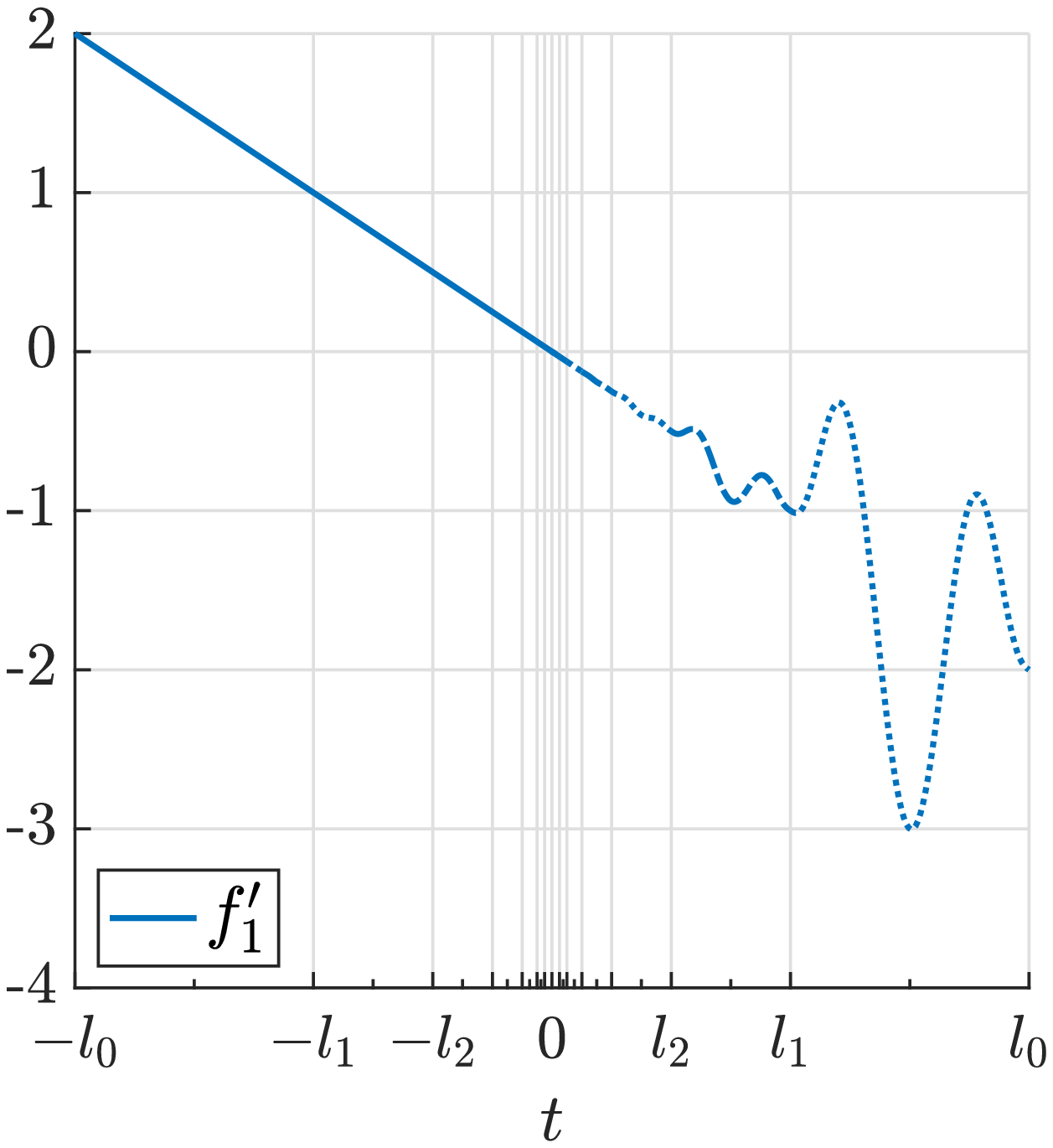}
		\includegraphics[scale=\scalederiv,trim=0mm 0mm 0mm 0mm,clip]{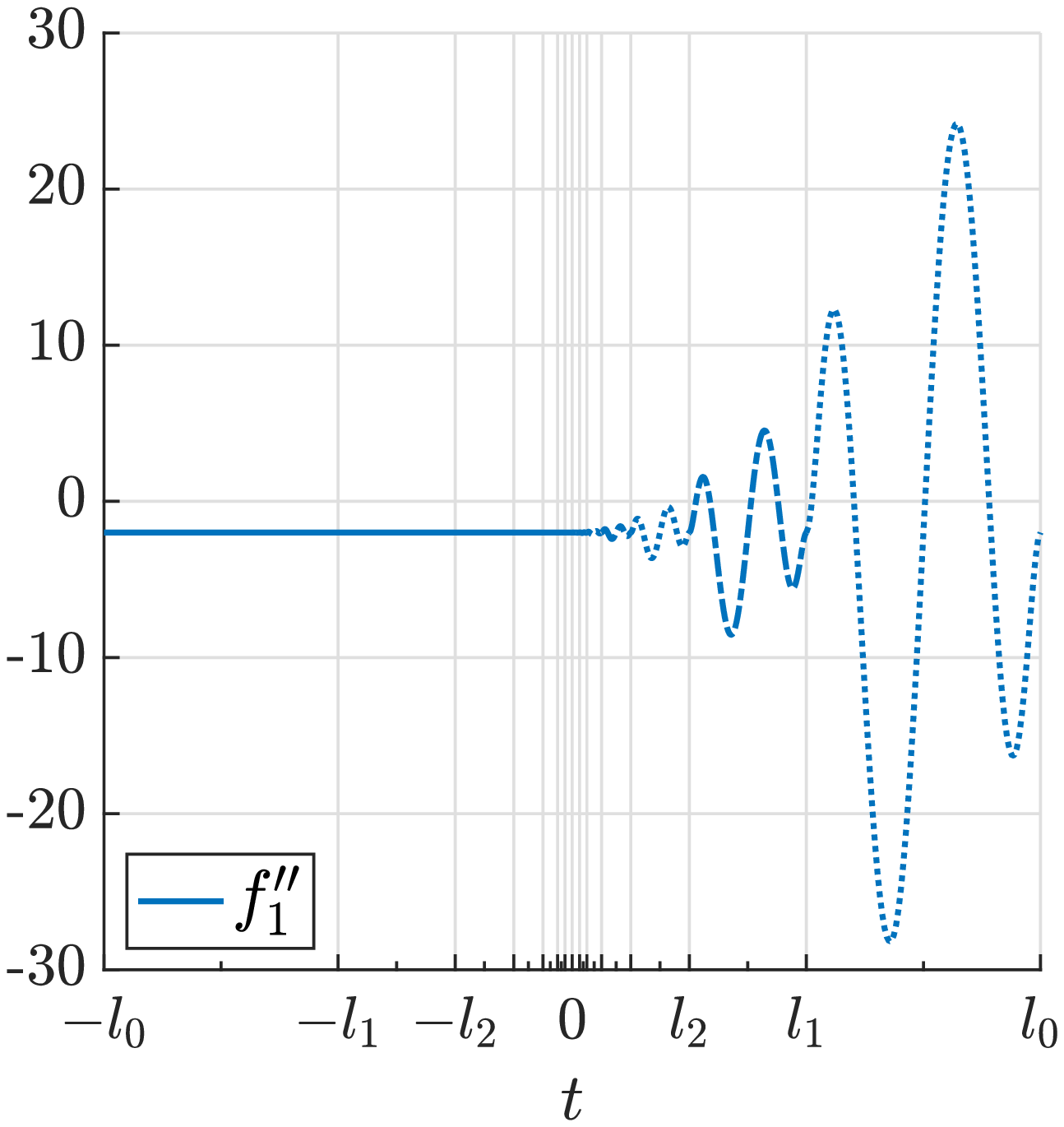}
		\includegraphics[scale=\scalederiv,trim=0mm 0mm 0mm 0mm,clip]{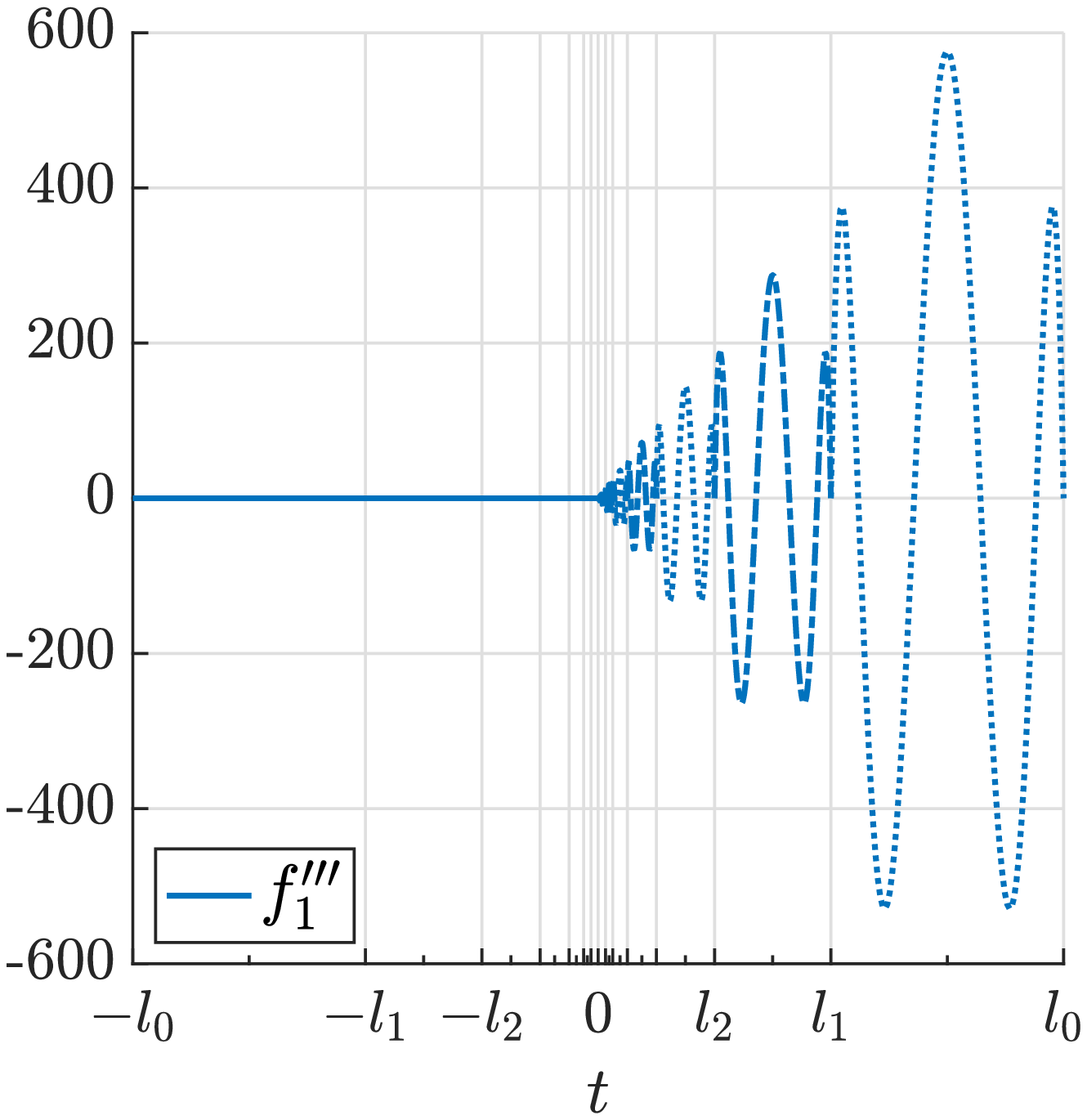}
		\caption{$s_k = 1 + 2^{-2k}$}
		\label{fig:fppp_C3}
	\end{subfigure}	
\caption{Plots of the first three derivatives of $f_1$ for two different sequences $\{s_k\}$; 
	their $-t^2$ parts are shown in solid, while their $p_k$ 
	parts are shown in dotted for $k$ even and dash-dot for $k$ odd.}
\label{fig:derivs}
\end{figure}

\begin{figure}[t]
	\begin{subfigure}{\textwidth}
		\centering
		\includegraphics[scale=\scalederiv,trim=0mm 0mm 0mm 0mm,clip]{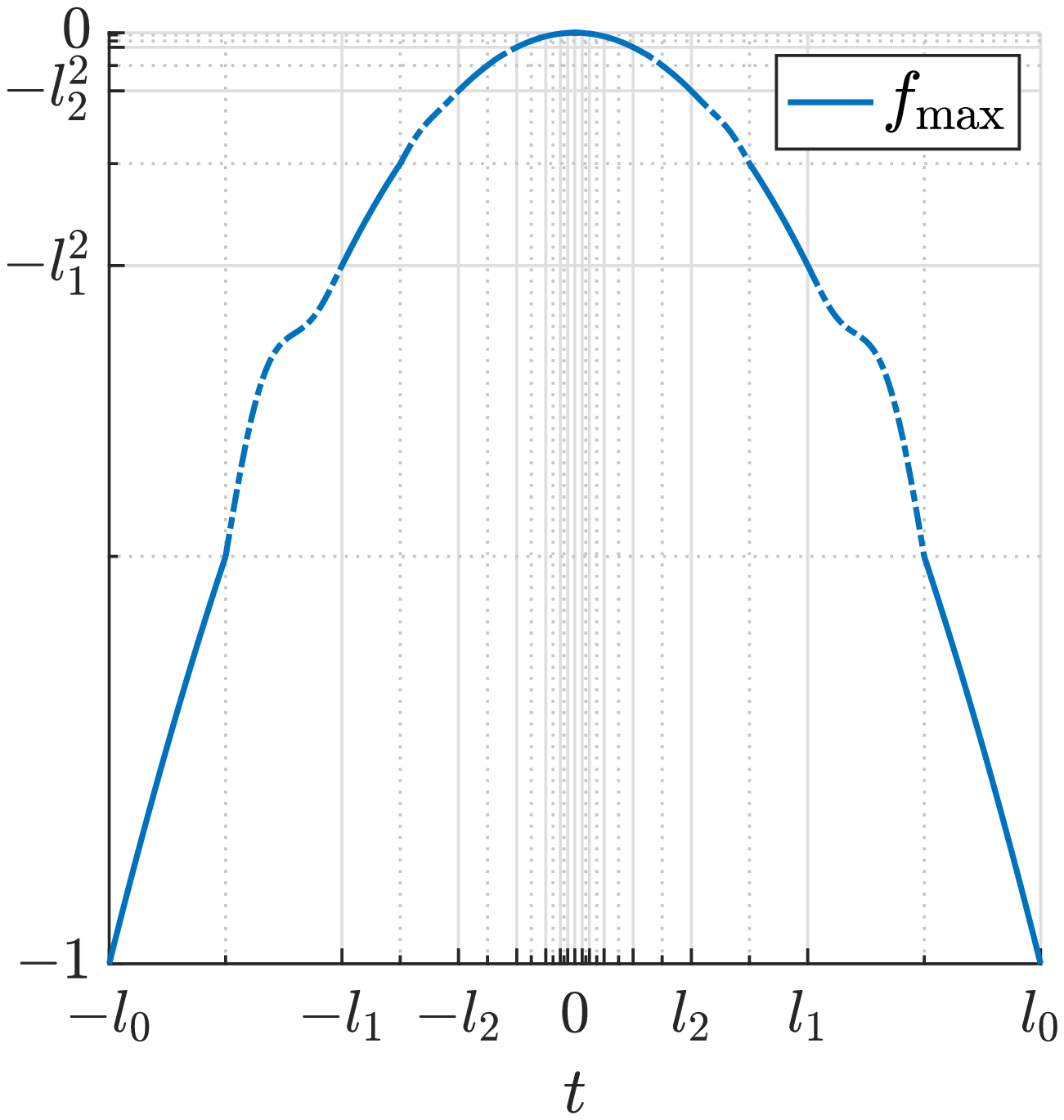}
		\includegraphics[scale=\scalederiv,trim=0mm 0mm 0mm 0mm,clip]{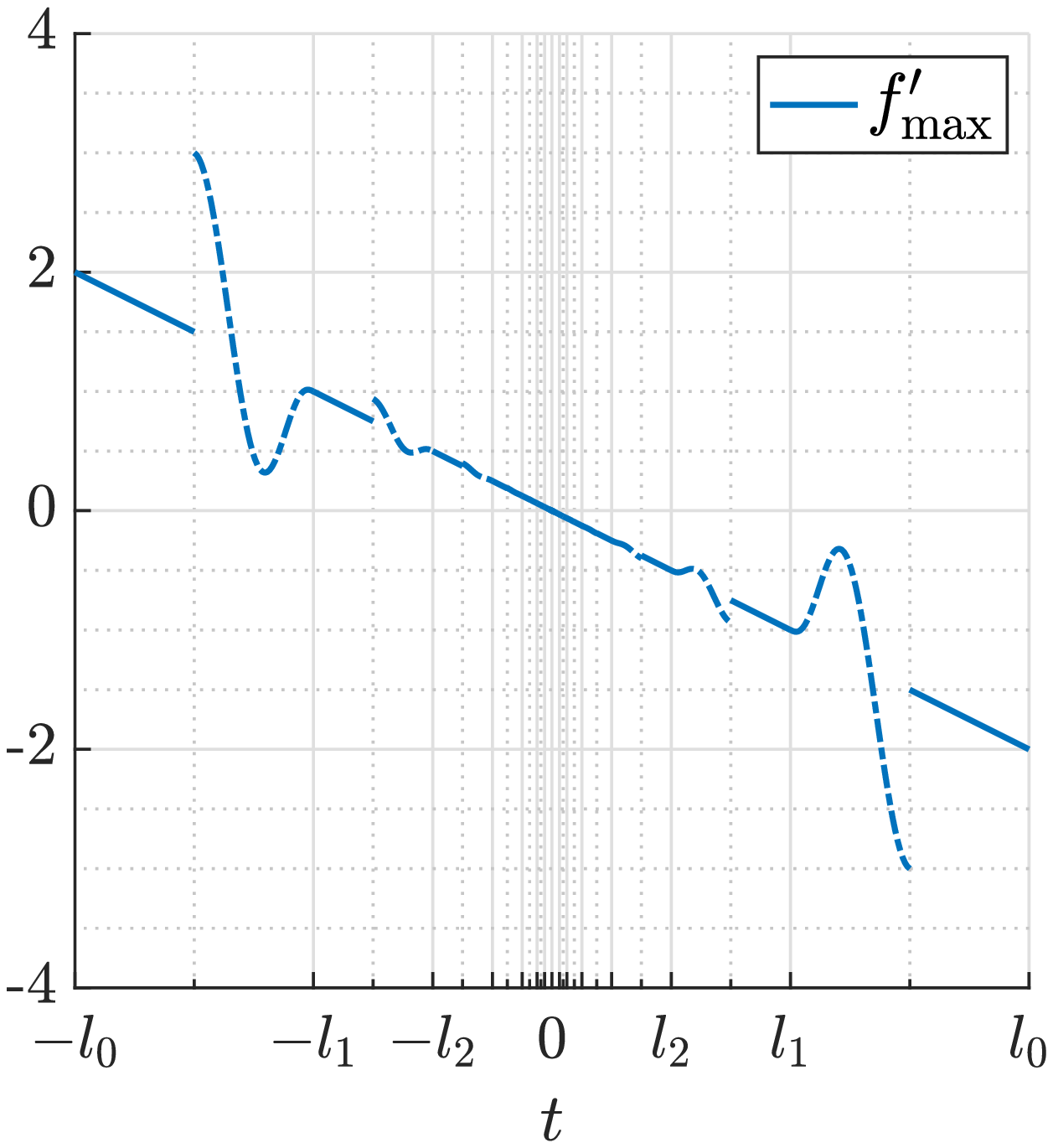}
		\includegraphics[scale=\scalederiv,trim=0mm 0mm 0mm 0mm,clip]{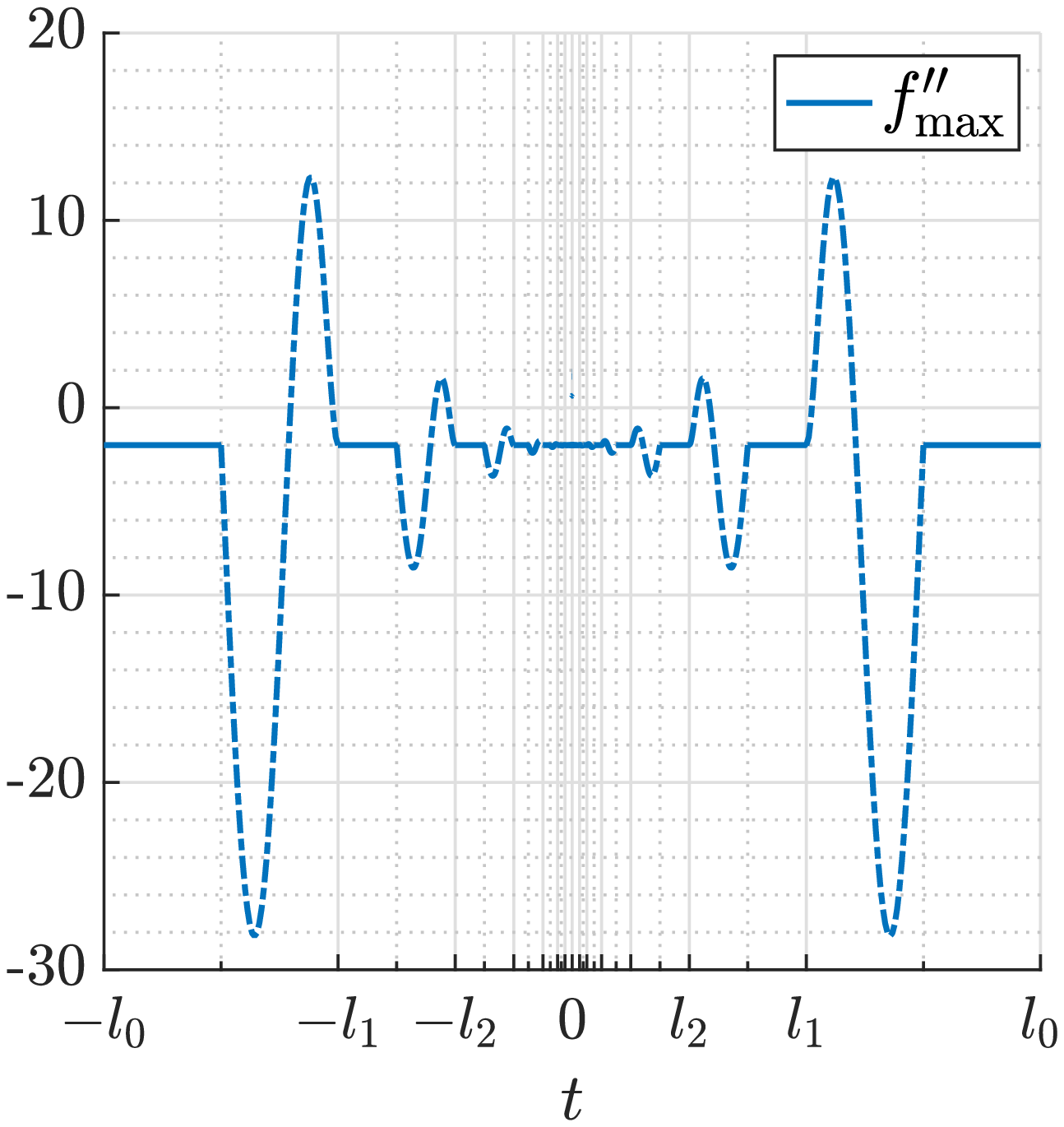}
	\end{subfigure}
\caption{Plots of $\gmax$ and its first and second derivatives; 
	their $-t^2$ parts are shown in solid, while their $p_k$
	parts are shown in dash-dot.}
\label{fig:gmax}
\end{figure}

\section{Smoothness of eigenvalue extrema and applications} 
\label{sec:eigs}
We will need the following well-known theorem:

\begin{theorem}[Rellich]
\label{thm:rell}
Let $H : \dom \to \H^{n}$ be an analytic Hermitian matrix family in one real variable.
Let $\opt \in \dom$ be given,
and let $H(\opt)$ have eigenvalues $\tilde \mu_{j}\in\R$, $j=1,\ldots,n$, not necessarily distinct. 
Then, for sufficiently small $|\eps|$, the eigenvalues of $H(\opt + \eps)$ can be expressed as convergent power series 
\beq\label{powser}
          \mu_{j}(\eps) = \tilde \mu_{j} +  \tilde \mu^{(1)}_{j}\eps + \tilde \mu^{(2)}_{j} \eps^{2} + \cdots, \quad j=1,\ldots, n.
\eeq
\end{theorem}

Theorem~\ref{thm:rell} is often stated as part of a deeper theorem of Rellich regarding power series expansion of the eigenvectors;
in comparison, the proof of~\eqref{powser} is significantly easier, using the theory of algebraic functions
to express the eigenvalues as fractional powers of $\eps$ and then arguing that, because $H$ is Hermitian, non-integral
fractional powers vanish~\cite[pp.~XIX--XX]{Kat82}.

We now apply Theorems~\ref{thm:rell} and \ref{thm:c2_lipschitz} to obtain smoothness results for 
eigenvalue extrema of univariate real analytic Hermitian matrix families, 
as well as analogous results for singular value extrema.
Subsequently, we discuss how these results are useful in several important applications.

\begin{theorem}
\label{thm:eig_tcd}
Let $H : \dom \to \H^{n}$ be an analytic Hermitian matrix family in one real variable
on an open domain $\dom \subseteq \R$, and let $\lmax:\H^{n}\to\R$ denote
algebraically largest eigenvalue.  
Then $\lmax \circ H$ is $\Ctwo$ with Lipschitz second derivative near all of its local maximizers.
\end{theorem}

\begin{proof}
Let $\opt \in \dom$ be any local maximizer of $\lmax \circ H$, 
with $H(\opt)$ having eigenvalues $\tilde\mu_{j}$.
By Theorem~\ref{thm:rell}, in a neighborhood of~$\opt$, the eigenvalues of $H(\opt + \eps)$
can be expressed as $\mu_{j}(\eps)$, $j=1,\ldots,n$, where the $\mu_{j}(\eps)$ are locally given by the power series \eqref{powser}.
Since $\lmax(H(\opt + \eps)) = \max_{j \in \{1,\ldots,n\}} \mu_j(\eps)$ with all the $\mu_j$ analytic,
we can apply Theorem~\ref{thm:c2_lipschitz} to these functions, completing the proof.
\end{proof}

\begin{remark}
The proof of Theorem~\ref{thm:eig_tcd} is essentially the same as the proof given by Boyd and Balakrishnan \cite{BoyB90}, presented differently and in
a more general context.
\end{remark}

\begin{corollary}
\label{cor:eigC2}
Let $H : \dom \to \H^{n}$ be an analytic Hermitian matrix family in one real variable
on an open domain $\dom \subseteq \R$. 
Then:
\begin{enumerate}[label=(\roman*),font=\normalfont,leftmargin=0.75cm]
	\item $\lmin \circ H$ is $\Ctwo$ near all of its local minimizers, 
		where $\lmin$ denotes algebraically smallest eigenvalue;
	\item $\srad \ \circ \ H$ is $\Ctwo$ near all of its local maximizers, 
		where $\srad$ denotes spectral radius $(\max(\lmax,-\lmin))$;
	\item $\sradin \circ H$ is $\Ctwo$ near all of its local minimizers at which the minimal value is nonzero,
	 	where $\sradin$ denotes inner spectral radius 
		$($0 if $H$ is singular, $\srad(H^{-1})^{{-1}}$ otherwise$)$.
\end{enumerate}
Furthermore, in each case the second derivative is Lipschitz near the relevant maximizers/minimizers.
\end{corollary}
\begin{proof}
Statements (i) and (ii) follow 
from applying Theorem~\ref{thm:eig_tcd}
to $-H$ and $\diag(H,-H)$, respectively.  
For (iii), apply (ii) to $\rho\circ H^{-1}$ and take the reciprocal.
\end{proof}

\begin{corollary}
\label{cor:singC2}
Let $A : \dom \to \C^{m \times n}$ be an analytic matrix family in one real variable
on an open domain $\dom \subseteq \R$, let $\smax$ denote largest singular value, and let 
$\smin$ denote smallest singular value, noting that the latter is nonzero if and only if the matrix has full rank.
Then:
\begin{enumerate}[label=(\roman*),font=\normalfont,leftmargin=0.75cm]
	\item $\smax \circ A$ is $\Ctwo$ near all of its local maximizers, and
	\item $\smin \circ A$ is $\Ctwo$ near all of its local minimizers at which the minimal value is nonzero.
\end{enumerate}
Furthermore, in each case the second derivative is Lipschitz near the relevant maximizers/minimizers.
\end{corollary}
\begin{proof}
If $m\geq n$, consider the real analytic Hermitian matrix family \mbox{$H: \dom\to\H^{n}$} defined by
\[
      H(t) = A(t)^{*} A(t) = \left (\Re A(t) - \imagunit\, \Im A(t) \right)\tp   \left (\Re A(t) + \imagunit\, \Im A(t) \right ),
\]
whose eigenvalues are the squares of the singular values of $A(t)$. Then (i) and (ii), respectively, follow from 
applying Corollary~\ref{cor:eigC2} (ii) and (iii), respectively, to $H(t)$, and then taking the square root.
If $n>m$, set $H(t)=A(t)A(t)^{*}$ instead.
\end{proof}

Corollary~\ref{cor:singC2}~(i) is the regularity result that Boyd and Balakrishnan established in~\cite{BoyB90}.
For Corollary~\ref{cor:singC2}~(ii), note that the assumption that the minimal value of $\smin \circ A$ is nonzero
is necessary; e.g., $\smin(t)$ is nonsmooth at its minimizer $t=0$.

\subsection{The $\Hinf$ norm}
This application was the original motivation for Boyd and Balakrishnan's work.
Let $A \in \C^{n \times n}$, $B \in \C^{n \times m}$, $C \in \C^{p \times n}$, and $D \in \C^{p \times m}$
and consider the linear time-invariant system with input and output:
\begin{subequations}
	\label{eq:lti}
	\begin{align}
	\dot x &= Ax + Bu, \\
	y &= Cx + Du.
	\end{align}
\end{subequations}
Assume that $A$ is asymptotically stable, i.e., its eigenvalues 
are all in the open left half-plane.
An important quantity in control systems engineering and model-order reduction is the $\Hinf$ norm of~\eqref{eq:lti},
which measures the sensitivity of the system to perturbation and can be computed
by solving the following optimization problem:
\begin{equation}
	\label{eq:hinf}
	\max_{\omega \in \R} \ \smax(G(\imagunit \omega)),
\end{equation}
where $G(\lambda) = C(\lambda I - A)^{-1}B + D$ is the transfer matrix
associated with~\eqref{eq:lti}. Even though there is only one real variable, finding the global maximum of this
function is nontrivial.

By extending Byer's breakthrough result on computing the distance to instability~\cite{Bye88},
Boyd et al.~\cite{BoyBK89} 
developed a globally convergent bisection method to solve~\eqref{eq:hinf}~to arbitrary accuracy.
Shortly thereafter, a much faster algorithm, 
 based on computing level sets of~$\smax(G(\imagunit \omega))$,
was independently proposed in~\cite{BoyB90} and \cite{BruS90},
with Boyd and Balakrishnan showing that this iteration converges quadratically~\cite[Theorem~5.1]{BoyB90}.
As part of their work, they showed that, with respect to the 
real variable~$\omega$, $\smax(G(\imagunit \omega))$ is $\Ctwo$ with Lipschitz second derivative
near any of its local maximizers~\cite[pp.~2--3]{BoyB90}.
Subsequently, this smoothness property has been leveraged to 
further accelerate computation of the $\Hinf$ norm~\cite{GenVV98,BenM18a}.

\subsection{The numerical radius}
Now consider the numerical radius of 
a matrix $A \in \C^{n \times n}$:
\begin{equation}
	r(A) = \max \{ |z | : z \in W(A)\},
\end{equation}
where $W(A) = \{ v^*Av : v \in \C^n, \|v\|_2 = 1\}$ is
the field of values (numerical range) of $A$.
Following~\cite[Ch.~1]{HorJ91}, the numerical radius can
be computed by solving either 
\begin{equation}
	r(A) = \max_{\theta \in [0,2\pi)} \lmax(H(\theta))
	\qquad \text{or}	\qquad 
	r(A) = \max_{\theta \in [0,\pi)} \srad(H(\theta)),
\end{equation}
where $H(\theta) = \tfrac{1}{2}\left( \eit A + \emit A^*\right)$.

In~\cite{MenO05}, Mengi and the second author proposed
the first globally convergent method guaranteed to compute $r(A)$ to arbitrary accuracy.
This was done by employing a level-set technique that converges to a global maximizer of~$\lmax \circ H$,
similar to the aforementioned method of~\cite{BoyB90,BruS90} for the $\Hinf$ norm,
and observing, but not proving, quadratic  convergence of the method.
Quadratic convergence was later proved by G\"urb\"uzbalaban in his PhD thesis~\cite[Lemma 3.4.2]{Gur12},
following the proof used in \cite{BoyB90}, showing that $\lmax \circ H$ is $\Ctwo$ near maximizers.

\subsection{Optimization of passive systems}
Let $\mathcal{M} = \{A,B,C,D\}$ denote the system~\eqref{eq:lti},
but now with $m=p$ and the associated transfer function $G$ being minimal and 
proper~\cite{ZhoDG96}.
Mehrmann and Van Dooren~\cite{MehV20}
 have recently shown that 
another important problem is to compute the maximal value $\Xi \in \R$ such that for all $\xi < \Xi$, 
the related system $\mathcal{M}_\xi = \{A_\xi,B,C,D_\xi\}$ is strictly passive\footnote{A 
strictly passive system is one whose stored energy is decreasing; for more a formal treatment, see \cite{MehV20}.},
where $A_\xi = A + \tfrac{\xi}{2}I_n$ and $D_\xi = D - \tfrac{\xi}{2}I_m$.
Letting $G_\xi$ be the transfer matrix associated with $\mathcal{M}_\xi$,
by~\cite[Theorem~5.1]{MehV20}, the quantity~$\Xi$ is the unique root of
\begin{equation}
	\label{eq:xi}
	\gamma(\xi) \coloneqq
	\min_{\omega \in \R} \lmin \left( G_\xi(\imagunit \omega)^* + G_\xi(\imagunit \omega) \right) = 0.
\end{equation}
Note that in contrast to the univariate optimization problems discussed previously, 
computing~$\Xi$ is a problem
in two real parameters, namely, $\xi$ and~$\omega$.  
In~\cite[section~5]{MehV20}, 
Mehrmann and Van Dooren introduced both a bisection algorithm to compute~$\Xi$,
and an apparently faster ``improved iteration" whose exact convergence properties were not established.
However, using the fact that~$\lmin$ in~\eqref{eq:xi} 
is $\Ctwo$ with Lipschitz second derivative near all its minimizers, as well as some other tools, 
the first author and Van Dooren
have since established a rate-of-convergence result for this ``improved iteration"
and also
 presented a much faster and more numerically reliable algorithm to compute~$\Xi$ with quadratic convergence~\cite{MitV21}.
 
\section{Concluding remarks}
\label{sec:conclusion}

We have shown that the maximum eigenvalue of a univariate real analytic Hermitian matrix family is unconditionally $\Ctwo$ near all its maximizers,
with Lipschitz second derivative.
Although the result is well known in the context of the maximum singular value of a transfer function, its generality and
simplicity have apparently not been fully appreciated.
We believe that this result and its corollaries may be useful in many applications, some of which were summarized in this paper.
We also investigated whether this smoothness property extends to max functions more generally, 
showing that
the pointwise maximum of a finite set of $q$-times continuously differentiable univariate functions must have zero derivative at a maximizer for $q=1$, but arbitrarily close to the maximizer, the derivative may not be defined,
even when $q=3$ and the maximizer is isolated.

All figures and the symbolically computed coefficients of $p_k$ 
given in Appendices~\ref{apdx:proofs} and~\ref{apdx:sk} 
can be generated by \matlab\ codes that are available upon request.

\appendix

\section{Proofs of Theorems~\ref{thm:f1_c3} and \ref{thm:f1_isolated}}
\label{apdx:proofs}

\begin{lemma}
\label{lem:pk_coeffs}
For $f_1$ defined in~\eqref{eq:f1}, if
 $s_k=1 + 2^{-2k}$, then the coefficients of the polynomial 
\mbox{$p_k(t) = \sum_{j=0}^9 c_j t^j$} are:
\begin{equation*}
	\label{eq:pk_coeffs}
	c_j = \begin{cases}
	z_j 2^{(j-4)k} - 1 & \text{if $j=2$} \\
	z_j 2^{(j-4)k} 	& \text{otherwise}
	\end{cases}
	\ \ \quad \text{with} \ \ \quad
	\begin{alignedat}{8}
	 	z_9 &= {}& -98304, & \quad & z_5 &= {}& -3631104, & \quad & z_1 &= {}& -61440, \\
		z_8 &= {}& 663552, && z_4 &= {}& 2585088, && z_0 &= {}& 4608. \\
		z_7 &= {}& -1966080, && z_3 &= {}& -1210368, && {}& \\
		z_6 &= {}& 3354624, && z_2 &= {}& 359424, && {}& 
	\end{alignedat}
\end{equation*}
\end{lemma}
\begin{proof}
The coefficients were computed symbolically in \matlab\ by solving 
the linear system defined by the generalized Vandermonde matrix and right-hand side
determining each $p_k$ in~\eqref{eq:f1}.
These formulas were also verified by comparing with numerical computations.
\end{proof}

\begin{proof}[Proof of Theorem~\ref{thm:f1_c3}]
Function $f_1$ defined in~\eqref{eq:f1} is clearly $\Cthree$ near any nonzero $t$, since 
our construction ensures that the first three derivatives of $p_k$ and $p_{k+1}$ match where they meet.
We must show that it is also $\Cthree$ at $t=0$.
First note that for the coefficients given in Lemma~\ref{lem:pk_coeffs}, we can replace their 
dependency on~$k$ with a dependency on~$t$ by 
using $k=-\ceil{\log_2 t}$.
Thus, $f_1$ can be written as follows:
\begin{equation}
	\label{eq:f1_no_k}
	f_1(t) = \begin{cases}
	\sum_{j=0}^9 \tilde c_j t^j & \text{if $t > 0$} \\
	-t^2 	& \text{if $t \in [-1,0]$} \\
	\end{cases}
\end{equation}
where $\tilde c_j$ is obtained by replacing $k$ in $c_j$ with $-\ceil{\log_2 t}$.

We begin by looking at the first derivative.
For $f_1^\prime$ to exist and be continuous at $t=0$, 
\begin{equation}
	\label{eq:f1_c1}
	f_1^\prime(0) 
	= \lim_{\eps \to 0^+} \frac{f_1(0 + \eps) - f_1(0)}{\eps}
	= \lim_{\eps \to 0^+} \frac{f_1(\eps)}{\eps} = 0
\end{equation}
must hold, i.e., the derivative from the right (over the $p_k$ pieces) 
must match the derivative from the left (over the $-t^2$ piece).
To show that \eqref{eq:f1_c1} holds, we show that each term in the sum in \eqref{eq:f1_no_k}
divided by $t$ goes to zero as $t \to 0^+$, 
i.e., that $\lim_{t \to 0^+} \tilde c_j t^{j-1} = 0$ for $j \in \{0,1,\ldots,9\}$.
It is obvious that this holds for $j=4$ since $c_4= \tilde c_4 = z_4$ is a fixed number.
To show the highest-order term ($j=9$) vanishes as $t \to 0^+$, 
we can make use of the fact that $0 < 2^{-\ceil{\log_2 t}} \leq 2^{-\log_2 t} = t^{-1}$
holds for all $t > 0$, i.e.,
\[
	\lim_{t \to 0^+} \big| z_9 2^{-5\ceil{\log_2 t}} t^8 \big| \leq 
	\lim_{t \to 0^+} \big| z_9 (t^{-1})^{5} t^8 \big| = 
	\lim_{t \to 0^+} \big| z_9 t^3 \big| = 0.
\]
Similar arguments show that $\lim_{t \to 0^+} \tilde c_j t^{j-1} = 0$ holds for $j \in \{5,6,7,8\}$.
Using the fact that 
\mbox{$0 < 2^{\ceil{\log_2 t}} \leq 2^{1 + \log_2 t} = 2t$} for all $t > 0$,
for $j=3$, we have that 
\[
	\lim_{t \to 0^+} \big| z_3 2^{\ceil{\log_2 t}} t^2 \big| \leq
	\lim_{t \to 0^+} \big| z_3 2t^3 \big| = 0,
\]
while for $j=2$ and $j=1$ we respectively have that  
\[
	\lim_{t \to 0^+} \big| z_2 2^{2\ceil{\log_2 t}} - 1 \big| t \leq 
	\lim_{t \to 0^+} \big| z_2 (2t)^2 - 1 \big| t =
	\lim_{t \to 0^+} \big| z_2 (4t^3 - t) \big| = 0.
\]
and
\[
	\lim_{t \to 0^+} \big| z_1 2^{3 \ceil{\log_2 t}} \big| \leq 
	\lim_{t \to 0^+} \big| z_1 (2t)^3 \big| = 0.
\]
Finally, for $j=0$, we have that 
\[
	\lim_{t \to 0^+}  \frac{z_0 2^{4 \ceil{\log_2 t}}}{t} \leq
	\lim_{t \to 0^+} \frac{z_0 (2t)^4}{t} = 0.
\]
Hence, we have shown that $f_1$ is at least $\Cone$ on its domain.

Analogously, for $f_1^{\prime\prime}$ to exist and be continuous at $t=0$, 
\begin{equation}
	\label{eq:f1_c2}
	f_1^{\prime\prime}(0) 
	= \lim_{\eps \to 0^+} \frac{f_1^\prime(0 + \eps) - f_1^\prime(0)}{\eps}
	= \lim_{\eps \to 0^+} \frac{f_1^\prime(\eps)}{\eps} = -2
\end{equation}
must hold.  
We have that 
\begin{equation}
	\label{eq:f1p_no_k}
	f_1^\prime(t) = \begin{cases}
	\sum_{j=1}^9 j \tilde c_j t^{j-1} & \text{if $t > 0$} \\
	-2t 	& \text{if $t \in [-1,0]$} \\
	\end{cases}
\end{equation}
and so we consider $\lim_{t \to 0^+} j \tilde c_j t^{j-2} $ for $j \in \{1,\ldots,9\}$,
i.e., the limit of each term in the sum in \eqref{eq:f1p_no_k} divided by $t$.
We show that for all but $j = 2$, these values goes to zero, while the $j=2$
value goes to $-2$ as $t\to0^+$.
For $j=9$, we have that
\[
	\lim_{t \to 0^+} \big| 9 z_9 2^{-5\ceil{\log_2 t}} t^7 \big| \leq
	\lim_{t \to 0^+} \big| 9 z_9 (t^{-1})^5 t^7 \big| = 0, 
\]
with similar arguments showing that $j\in \{5,6,7,8\}$ values also diminish to zero.
For $j=4$, we simply have $\lim_{t\to0^+} 4z_4 t^2 = 0$.
For $j=3$, 
\[
	\lim_{t \to 0^+} \big| 3 z_3 2^{\ceil{\log_2 t}} t \big| \leq
	\lim_{t \to 0^+} \big| 3 z_3 (2t) t \big| 
	= 0.
\]
For $j=2$, we have that
\[
	\lim_{t \to 0^+} 2 (z_2 2^{2\ceil{\log_2 t}} - 1) =
	\lim_{t \to 0^+} 2 z_2 (2^{\ceil{\log_2 t}})^2 - 2 
	= -2.
\]
Lastly, for $j=1$, we have that 
\[
	\lim_{t \to 0^+} \frac{\big| z_1 2^{3\ceil{\log_2 t}} \big| }{t} \leq
	\lim_{t \to 0^+} \frac{\big| z_1 (2t)^3 \big|}{t} = 0,
\]
and so we have now shown that $f_2$ is at least $\Ctwo$ on its domain.

Finally, for $f_1^{\prime\prime\prime}$ to exist and be continuous at $t=0$, 
\begin{equation}
	\label{eq:f1_c3}
	f_1^{\prime\prime\prime}(0) 
	= \lim_{\eps \to 0^+} \frac{f_1^{\prime\prime}(0 + \eps) - f_1^{\prime\prime}(0)}{\eps}
	= \lim_{\eps \to 0^+} \frac{f_1^{\prime\prime}(\eps) + 2}{\eps} = 0
\end{equation}
must hold. 
We have that 
\begin{equation}
	\label{eq:f1pp_no_k}
	f_1^{\prime\prime}(t) = \begin{cases}
	\sum_{j=2}^9 j(j-1) \tilde c_j t^{j-2} & \text{if $t > 0$} \\
	-2 	& \text{if $t \in [-1,0]$} \\
	\end{cases}
\end{equation}
and so we consider $\lim_{t \to 0^+} j (j-1) \tilde c_j t^{j-3} $ for $j \in \{2,\ldots,9\}$,
i.e., the limit of each term in the sum in \eqref{eq:f1pp_no_k} divided by $t$.
For $j \in \{5,6,7,8,9\}$, we again have similar arguments showing that the corresponding values vanish, 
so we just show the $j=9$ case, which follows because
\[
	\lim_{t \to 0^+} \big| 72 z_9 2^{-5\ceil{\log_2 t}} t^6 \big| \leq 
	\lim_{t \to 0^+} \big| 72 z_9 (t^{-1})^5 t^6 \big|  = 0.
\]
Again, it is clear that the value for $j=4$ vanishes.
For $j=3$, we have that 
\[
	\lim_{t \to 0^+} \big| 6 z_3 2^{\ceil{\log_2 t}} \big| \leq 
	\lim_{t \to 0^+} \big| 6 z_3 (2t) \big|
 	= 0.
\] 
Finally, for $j=2$, we can rewrite \eqref{eq:f1_c3} as follows, making use of these aforementioned limits which vanish and replacing $\eps$ by $t$,
to obtain a limit only involving the $j=2$ term:
\begin{align*}
	f_1^{\prime\prime\prime}(0) 
	= \lim_{t \to 0^+} \frac{f_1^{\prime\prime}(t) + 2}{t} 
	&= \lim_{t \to 0^+} \frac{2 (z_2 2^{2\ceil{\log_2 t}} - 1) + 2}{t}  \\
	&= \lim_{t \to 0^+} \frac{2 z_2 (2^{\ceil{\log_2 t}})^2}{t}
	\leq \lim_{t \to 0^+} \frac{2 z_2 (2t)^2}{t} = 0.
\end{align*}
Thus, $f_1$ is indeed $\Cthree$ on its domain.
\end{proof}

\begin{proof}[Proof of Theorem~\ref{thm:f1_isolated}]
Since $l_k$ is a power of two, we can rewrite the derivative 
of $p_k$, i.e., 
 $
 	p_k^\prime(t) = \sum_{j=1}^9 j c_j t^{j-1},
 $
as a function of $\zeta \in [1,2]$:
\begin{align*}
	\tilde p_k^\prime(\zeta) 
	= \sum_{j=1}^9 j c_j (l_{k+1}\zeta)^{j-1}
	&= \sum_{j=1}^9 \frac{j c_j}{2^{(k+1)(j-1)}}\zeta^{j-1} 
	=  \frac{2 (z_2 2^{-2k} - 1)}{2^{k+1}} \zeta +  
		\sum_{\substack{j=1 \\ j\neq 2}}^9 \frac{j  z_j 2^{(j - 4)k}}{2^{(k+1)(j-1)}} \zeta^{j-1} \\
	& = \frac{ z_2  - 2^{2k}}{2^{3k}} \zeta +  
		\sum_{\substack{j=1 \\ j\neq 2}}^9 \frac{j  z_j 2^{1 - j}}{2^{3k}} \zeta^{j-1} 
	= \frac{1}{2^{3k}} \Bigg( (z_2  - 2^{2k}) \zeta + 
		\sum_{\substack{j=1 \\ j\neq 2}}^9 \tilde z_j \zeta^{j-1} \Bigg),
\end{align*}
where $\tilde z_j = j z_j 2^{1 - j}$.
From Lemma~\ref{lem:pk_coeffs}, we see that $z_2 - 2^{2k} < 0$ for all $k \geq 10$,
while for any~$k$, we have that $\tilde z_j < 0$ for $j \in \{1,3,5,7,9\}$ and 
$\tilde z_j > 0$ for $j \in \{4,6,8\}$.
Since $\zeta \in [1,2]$, an upper bound for $\tilde p_k^\prime$ can be obtained by evaluating 
its negative terms at $\zeta=1$ and its positive terms at $\zeta = 2$, i.e., for all $k \geq 10$ 
and any $\zeta \in [1,2]$, we have that
\[
	 \tilde p_k^\prime(\zeta) \leq 
	 \frac{1}{2^{3k}} \Bigg( (z_2  - 2^{2k}) + 
	 	\sum_{j\in \{ 1,3,5,7,9\} } \tilde z_j + 
		\sum_{j\in \{4,6,8\}} \tilde z_j 2^{j-1} \Bigg).
\]
For $k\geq 13$, the upper bound on the derivative is negative.
Thus, for $k\geq 13$,  $\tilde p_k^\prime(\zeta) < 0$ for any $\zeta \in [1,2]$, so
 $p_k$ must be decreasing. Consequently, the $t=0$ maximizer of $f_1$ is isolated.  
Finally, it immediately follows that the $t=0$ maximizer of $\gmax=\max(f_{1},f_{2})$ is also isolated.
\end{proof}

\section{Why $s_k=1+2^{-k}$ is insufficient to make \eqref{eq:f1} a $\Cthree$ function}
\label{apdx:sk}
For $s_k=1 + 2^{-k}$, symbolic computation shows that 
the coefficients of $p_k(t) = \sum_{j=0}^9 c_j t^j$ are:
\[	
	c_j = \begin{cases}
	z_j 2^{(j-3)k} - 1 & \text{if $j=2$} \\
	z_j 2^{(j-3)k} 	& \text{otherwise}
	\end{cases}
\]
where the integers $z_j$ remain the same as given in Lemma~\ref{lem:pk_coeffs}. 
To see if \eqref{eq:f1_c3} still holds for this new choice of $s_k$
we look at $\lim_{t \to 0^+} j (j-1) \tilde c_j t^{j-3} $ for $j \in \{2,\ldots,9\}$.
However, now none of the individual limits vanish.
For example, for $j=9$, we have that 
\[
	\lim_{t \to 0^+} \big| 72 z_9 2^{-6\ceil{\log_2 t}} t^6 \big| \geq
	\lim_{t \to 0^+} \big| 72 z_9 (2^{-1}t^{-1})^{6} t^6 \big| =  \tfrac{9}{8} \big|z_9 \big| \neq 0,
\]
where we have used the fact that $0 <  2^{-1}t^{-1} = 2^{-1 - \log_2 t} \leq 2^{-\ceil{\log_2 t}}$;
similarly, the limits for $j\in\{4,5,6,7,8\}$ do not vanish either.
For $j=3$, we simply have that 
$\lim_{t \to 0^+} 6 z_3 = 6z_3 \neq 0$.
Finally, even if all of the terms considered above were to vanish and we substitute in the value for $j=2$ into \eqref{eq:f1_c3},
we nevertheless would end up attaining another limit that does not vanish:
\[
	\lim_{t \to 0^+} \frac{2 z_2 2^{\ceil{\log_2 t}}}{t} 
	\geq \lim_{t \to 0^+} \frac{2 z_2 (t)}{t} = 2 z_2 \neq 0.
\]	
The only remaining way that \eqref{eq:f1_c3} could hold is if all of these non-vanishing terms
cancel, but from our experiments (see Figure~\ref{fig:fppp_notC3}), we know this is not the case.